\documentclass{article}

\usepackage[margin=1in]{geometry}
\usepackage{amssymb,amsmath,amsthm, bbm}
\usepackage{color, xcolor}
\usepackage{graphicx}
\usepackage{multirow}
\usepackage{array}
\usepackage{hyperref}
\usepackage{fancyhdr}
\usepackage{listings}
\usepackage{enumitem}
\usepackage{mathtools}
\usepackage{esvect}
\usepackage{mathrsfs}
\usepackage{authblk}
\usepackage[bottom]{footmisc}
\usepackage[sorting=ynt]{biblatex}

\hypersetup{
    colorlinks = true,
    urlcolor = cyan,
    linkcolor = blue,
    citecolor = blue,
}

\DeclareMathOperator{\Span}{span}

\newcommand{\concat}{\|}

\theoremstyle{plain}
\newtheorem{theorem}{Theorem}[section]
\newtheorem{corollary}[theorem]{Corollary}
\newtheorem{proposition}[theorem]{Proposition}
\newtheorem{lemma}[theorem]{Lemma}

\theoremstyle{definition}
\newtheorem{definition}[theorem]{Definition}

\theoremstyle{remark}
\newtheorem{remark}[theorem]{Remark}

\title{On Calculating the Chromatic Symmetric Function}

\author[1]{Nima Amoei Mobaraki\footnote{nima.amoei.mobaraki@gmail.com}}
\author[1,2]{Yasaman Gerivani\footnote{gerivaniyasaman@gmail.com}}
\author[1]{Sina Ghasemi Nezhad\footnote{sina.ghaseminejad@gmail.com}}

\affil[1]{\small Department of Mathematical Sciences, Sharif University of Technology, Tehran, Iran}
\affil[2]{\small Department of Chemical and Petroleum Engineering, Sharif University of Technology, Tehran, Iran}

\date{}

\begin{document}

\maketitle


\begin{abstract}
    This paper investigates methods for calculating the chromatic symmetric function (CSF) of a graph in chromatic-bases and the $m_\lambda$-basis. Our key contributions include a novel approach for calculating the CSF in chromatic-bases constructed from forests and an efficient method for determining the CSF in the $m_\lambda$-basis. As applications, we present combinatorial proofs for two known theorems that were originally established using algebraic techniques. Additionally, we demonstrate that an algorithm introduced by Aliste-Prieto, de Mier, Orellana, and Zamora can be viewed as a case of our proposed method.
\end{abstract}


\section{Introduction} \label{section:intro}

The chromatic symmetric function (CSF) for a simple graph was first introduced by Stanley in 1995 as follows:

\begin{definition} \label{def:csf}
    \cite{stanley1995symmetric}
    Let $G$ be a finite graph with vertex set $V(G)$. The \emph{chromatic symmetric function}, denoted $X_G(x)$, is given by:
    $$
        X_G(x) \triangleq \sum_{\kappa} \prod_{v \in V(G)} x_{\kappa(v)},
    $$
    where the sum runs over all proper colorings $\kappa : V(G) \to \mathbb{Z}^{+}$. A coloring $\kappa$ is proper if adjacent vertices receive distinct colors.
\end{definition}

Stanley explored numerous properties of the CSF, demonstrating its connection to classical bases in symmetric function theory. A key observation is that setting $x_i = 1$ for $1 \leq i \leq k$ and $x_j = 0$ for $j > k$ recovers the chromatic polynomial $\chi_G(k)$, which counts the valid colorings of $G$ using exactly $k$ colors. The CSF has since become a topic of significant interest, with contributions from various researchers; see, e.g., \cite{martin2008distinguishing, orellana2014graphs, cho2015chromatic, aliste2017trees, aliste2023marked, chan2024graph, sagan2024chromatic}.

Two major conjectures have shaped the study of the CSF: the $e$-\emph{Positivity Conjecture}~\cite{stanley1995symmetric}, which asserts that if a poset is $(3+1)$-free, then its incomparability graph has a nonnegative expansion in terms of elementary symmetric functions, and the \emph{Tree Isomorphism Conjecture}, which proposes that the CSF uniquely identifies trees up to isomorphism. The first of these conjectures was recently resolved by Hikita~\cite{hikita2024proof}, while the second remains open but has been verified for trees with fewer than 30 vertices \cite{heil2018algorithm} and specific tree families \cite{martin2008distinguishing, loebl2019isomorphism}.

A useful strategy in studying the CSF has been comparing it to other graph polynomials, such as the \emph{subtree polynomial} of Chaudhary and Gordon~\cite{martin2008distinguishing}, the $W$-polynomial of Noble and Welsh~\cite{noble1999weighted}, and the \emph{Tutte polynomial}. For instance, Loebl and Sereni \cite{loebl2019isomorphism} established that caterpillar trees are uniquely determined by the CSF via their $W$-polynomial. This connection is particularly noteworthy since the CSF and the $W$-polynomial coincide for trees~\cite{noble1999weighted}, suggesting an alternative approach to the Tree Isomorphism Conjecture via the $W$-polynomial.

Beyond these applications, the CSF has been studied in contexts such as Schur positivity, which ties it to the representation theory of symmetric and general linear groups. Extensions of the CSF have also been explored, including the $q$-quasisymmetric function of Shareshian and Wachs~\cite{shareshian2016chromatic}, non-commutative analogues~\cite{gebhard2001chromatic}, and variations that incorporate root structures~\cite{loehr2024rooted} or weighted edges~\cite{crew2020deletion}.

Our work, along with recent developments, makes use of linear combination relations. In Section \ref{section:main-1}, we employ a relation introduced by Orellana and Scott~\cite{orellana2014graphs}. Another notable relation is the \emph{Deletion-Near-Contraction} (DNC) relation, recently proposed by Aliste-Prieto, De-Mier, Orellana, and Zamora~\cite{aliste2023marked}, and also utilized in \cite{orellana2014graphs}.

Another key direction in CSF research was initiated by Cho and van Willigenburg~\cite{cho2015chromatic}, who introduced graph-based bases for symmetric functions, referred to as \emph{chromatic bases}. These bases are constructed from sequences of connected graphs $\{G_n\}_{n \geq 1}$ where each $G_n$ has exactly $n$ vertices. The appeal of chromatic bases stems from their direct relation to graph structures, which enhances our understanding of the CSF's behavior. Among these, the \emph{star-basis}, introduced by Gonzalez, Orellana, and Tomba, has been of particular interest. In Section \ref{section:main-1}, we investigate computational methods for expressing the CSF in a chromatic-basis and discuss its implications.

The central theme of this paper is the concept of \emph{relativizing graphs}, which we believe holds considerable potential, particularly in settings where algebraic structures are imposed on graphs and require interrelations. With this perspective, in Section \ref{section:main-1}, we propose a framework for computing the CSF of forests within a forest-basis, defined as a chromatic-basis comprising only forests. In Section \ref{section:main-2}, we establish a methodology for relating a given graph $G$ to smaller graphs and introduce a linear relation for determining the coefficients of $X_G$ in the $m_\lambda$-basis.

Additionally, in Section \ref{section:main-1}, we define an operation on graphs inspired by a relation from \cite{orellana2014graphs} and demonstrate its utility in constructing paths between graphs, thereby enabling both combinatorial and algebraic connections. In particular, we develop a systematic approach for forming transitions between graphs such that their CSFs can be expressed in terms of each other and their smaller subgraphs. The smaller graphs, in this context, serve as the difference between two larger graphs, allowing us to control and manipulate these differences to establish novel relations, similar to the DNC-relation. Through this approach, we rigorously prove propositions concerning the interdependence of CSFs for graphs that are connected through such transitions. As applications, we provide combinatorial proofs for two known theorems: the chromatic-basis theorem of Cho and van Willigenburg~\cite{cho2015chromatic} and the equivalence of the $U$-polynomial and the CSF over forests, as established by Loebl and Sereni~\cite{noble1999weighted}. Furthermore, we illustrate how the algorithm from~\cite{aliste2023marked} can be viewed as an interesting case of our method.

In Section \ref{section:main-2}, we adopt an approach inspired by the \emph{Reconstruction Conjecture}, leveraging morphism counts from smaller graphs to $G$ to derive a relation between their CSFs and $X_G$. Since it is known that the Reconstruction Conjecture holds for trees, we aim to utilize CSFs to recover subgraphs from a given tree and explore whether this methodology can extend to solving the Tree Conjecture. To this end, we establish linear relations between the CSF of a graph and the CSFs of its smaller subgraphs, enabling us to compute the number of subgraphs of a given graph systematically. Furthermore, we examine the vector space spanned by CSFs of graphs, shedding light on structural properties that emerge from this perspective.


\section{Preliminaries} \label{section:pre}

For the basic concepts in graph theory, we refer the reader to any standard introductory text, such as~\cite{west2001introduction}. For convenience, we outline some of the key definitions here. For information regarding the reconstruction conjecture, we refer the reader to~\cite{harary2006survey}.

A \emph{graph} $ G $ is an ordered pair $ (V, E) $, where $ V $ is a finite set of \emph{vertices} and $ E $ is a multiset of \emph{edges}. Each edge $ e $ is an unordered pair of two distinct vertices $ u $ and $ v $, which are called the \emph{endpoints} of $ e $. Note that our graphs are \emph{simple}, meaning that we do not allow for parallel edges and loops. A \emph{loop} in $ G $ is an edge connecting a vertex to itself, while two edges are \emph{parallel} if they connect the same vertices.

Let $ \mathbb{Z}^{+}$ denote the set of positive integers. A \emph{proper coloring} of $ G $ is a function $ \kappa : V \rightarrow \mathbb{Z}^{+} $ such that $ \kappa(u) \neq \kappa(v) $ whenever $ u $ and $ v $ are adjacent. If the image of $ \kappa $ is restricted to the subset $ \left\{1, 2, \ldots, k\right\} $, we say that $ \kappa $ is a $k$-\emph{coloring} of $ G $. The total number of $ k $-colorings of $ G $ is denoted by $ \chi_G(k) $, and $ \chi_G(k) $ is known to be a polynomial in $ k $, called the \emph{chromatic polynomial} of $ G $.

A \emph{subgraph} $G^{\prime} \subseteq G$ of a graph $G = (V, E)$ is a graph $G^{\prime} = (V^{\prime}, E^{\prime})$ such that $V^{\prime} \subseteq V$ and $E^{\prime} \subseteq E$, and for each edge $e\in E^\prime$, the endpoints of $e$ are included in $V^{\prime}$. A subgraph is said to be \emph{induced} by the vertex set $V^{\prime} \subseteq V$ if every edge in $E$ with both endpoints in $V^{\prime}$ is also included in $E^{\prime}$.

\begin{definition} \label{def:u-poly}
    \cite{noble1999weighted}
    The $U$-\emph{polynomial} $U_G(x, y)$, which is a special case of the $W$-polynomial, of a graph $G$ with vertex set $V(G)$ and edge set $E(G)$ is a multivariate polynomial defined as follows:

    Let $x = (x_1, x_2, \ldots)$ and $y$ be commuting indeterminates. For a subset $A \subseteq E(G)$, let $k(G|_A)$ be the number of connected components of the subgraph induced by $A$. The $U$-polynomial is given by:
    $$
        U_G(x, y) \triangleq \sum_{A \subseteq E(G)} \left( \prod_{i=1}^{|V(G)|} x_i^{c_i(A)} \right) (y-1)^{|A|-k(G|_A)} ,
    $$
    where $c_i(A)$ denotes the number of components of $G|_A$ with $i$ vertices.
\end{definition}

\begin{remark} \label{def:u-poly-forest}
    As demonstrated in~\cite{aliste2017trees}, the $U$-polynomial of a forest $F$ can be expressed as follows:
    $$
        U_F (x) = \sum_{A \subseteq E(F)} x_{\lambda(A)} = \sum_{A \subseteq E(F)} x_{\lambda(E(F) \setminus A)}.
    $$
    Here, $\lambda (A) = (\lambda_1, \ldots, \lambda_\ell)$ is the partition induced by the number of vertices on each connected component of $G|_A$, and $x_{\lambda(A)}$ is the monomial $x_{\lambda_1} \cdots x_{\lambda_\ell}$. Also, as defined in~\cite{aliste2017trees}, the \emph{restricted $U$-polynomial} of a forest $F$, given a positive integer $k$, is defined as:
    $$
        U_F^k (x) \triangleq \sum_{\substack{A \subseteq E(F),\\ |A| \leq k}} x_{\lambda (A)}.
    $$
\end{remark}

\begin{definition} \label{def:csf-vec}
    A \emph{partition} of $n \in \mathbb{Z}^{+}$, denoted by $\lambda \triangleq (\lambda_1, \lambda_2, \ldots, \lambda_\ell)$, is a sequence of positive integers $\lambda_i \in \mathbb{Z}^{+}$ such that $\sum_{i=1}^{\ell} \lambda_i = n$. This is written as $\lambda \vdash n$. Throughout this paper, we assume that the parts $\lambda_i$ are arranged in descending order.

    For $n \in \mathbb{Z}^{+}$, let the \emph{partition number} $p(n)$, represent the number of distinct partitions of $n$.

    For a partition $\lambda = (\lambda_1, \lambda_2, \ldots, \lambda_\ell)$, define the \emph{length of $\lambda$}, denoted by $\ell(\lambda)$, the number of parts in $\lambda$.
\end{definition}

\begin{definition} \label{def:part-func}
    Let $G$ be a graph with $n$ vertices. The \emph{partition function} of $G$ is defined as:
    $$
        \text{Part}(G) \triangleq \lambda, \ \text{where} \ \lambda = (\lambda_1, \lambda_2, \ldots, \lambda_\ell) \vdash n,
    $$
    with $\lambda_i$ representing the size of the $i$-th connected component of $G$.
\end{definition}

\begin{definition}
    Let $\lambda$ and $\mu = (\mu_1, \mu_2, \ldots, \mu_\ell)$ be two partitions of $n \in \mathbb{Z}^{+}$. We say that $\mu$ is \emph{finer} than $\lambda$, or equivalently, that $\lambda$ is \emph{coarser} than $\mu$, denoted by $\mu \leq \lambda$, if $\lambda$ can be obtained by replacing some of the parts $\mu_i$ in $\mu$ with their sum.
\end{definition}

\begin{definition}
    Assume that $\lambda \vdash n $ is a partition such that $\lambda = (\lambda_1,\lambda_2,\ldots,\lambda_k,1,1, \ldots, 1)$, with $\lambda_k > 1$. We call $\lambda^{\ast} \triangleq (\lambda_1,\lambda_2,\ldots,\lambda_k)$ the reduced form of $\lambda$. Furthermore, two partitions $\lambda^1$ and $\lambda^2$ are called equivalent if their reduced forms are equal. In addition, a partition $\lambda^1 \vdash n$ can be $s$-\emph{reduced} if there is a partition $\lambda^2 \vdash s$ such that $s\leq n$ and $\lambda^1$ and $\lambda^2$ are equivalent. $\lambda ^2$ is called the $s$-\emph{reduced form} of $\lambda ^1$.
\end{definition}

\begin{definition} \label{def:graph-vec-space}
    Let $A$ be a set of graphs. We define \emph{graph vector space} of $A$, denoted by $\mathcal{V}(A)$, as:
    $$
        \mathcal{V}(A) \triangleq \Span\left\{ X_G \mid G \in A \right\}.
    $$
\end{definition}

\begin{theorem} \label{thm:tri-rec-formula}
    \cite{orellana2014graphs}
    Let $G$ be a graph where $ e_1, e_2, e_3 \in E(G) $ form a triangle. Define the following subgraphs:
    \begin{itemize}
        \item $ G_{2,3} = (V(G), E(G) \setminus \left\{e_1\right\}) $,
        \item $ G_{1,3} = (V(G), E(G) \setminus \left\{e_2\right\}) $,
        \item $ G_3 = (V(G), E(G) \setminus \left\{e_1, e_2\right\}) $.
    \end{itemize}
    Then the CSF $ X_G $ of $G$ can be expressed as:
    $$
        X_G = X_{G_{2,3}} + X_{G_{1,3}} - X_{G_3}.
    $$
\end{theorem}

\begin{corollary} \label{cor:triless-rec-formula}
    \cite{orellana2014graphs}
    Let $G$ be a graph with the adjacent edges $ e_1 = vv_1 $, $ e_2 = vv_2 $, and $ e_3 = v_1v_2 \notin E(G) $; that is, $ e_1 $ and $ e_2 $ meet at the vertex $ v $, but there is no edge connecting $ v_1 $ to $ v_2 $. Define the following graphs:
    \begin{itemize}
        \item $ G_{1,3} = (V(G), (E(G) \setminus \left\{e_2\right\}) \cup \left\{e_3\right\}) $,
        \item $ G_{2,3} = (V(G), (E(G) \setminus \left\{e_1\right\}) \cup \left\{e_3\right\}) $,
        \item $ G_1 = (V(G), E(G) \setminus \left\{e_2\right\}) $,
        \item $ G_3 = (V(G), (E(G) \setminus \left\{e_1, e_2\right\}) \cup \left\{e_3\right\}) $.
    \end{itemize}
    Then the CSF $ X_G $ of $G$ can be expressed as:
    $$
        X_G = X_{G_{2,3}} + X_{G_1} - X_{G_3}.
    $$
\end{corollary}

\begin{definition} \label{def:mo-vec}
    Given two graphs $G$ and $H$, such that $|V(G)| > |V(H)|$. We define the \emph{induced subgraph count} of $H$ in $G$, denoted by $\binom{G}{H}$, as:
    $$
        \binom{G}{H} \triangleq \left|\left\{G^\prime \mid G^\prime \text{ is an induced subgraph of $G$ which is isomorphic to $H$}\right\}\right|.
    $$
\end{definition}

\begin{definition} \label{def:graph-family}
	We define three families of graphs as follows:
	\begin{itemize}
		\item Define $\mathcal{G}_n$ to be the family of all graphs with $n$ vertices.
		\item Define $\mathcal{F}_n$ to be the family of all forests with $n$ vertices.
		\item Define $\mathcal{T}_n$ to be the family of all trees with $n$ vertices.
	\end{itemize}
\end{definition}

\begin{definition} \label{lambda-graphs}
Let $\lambda = (\lambda_1, \lambda_2, \ldots, \lambda_\ell)$ be a partition of $n$. The following are some notable graphs constructed based on this partition, which will be used frequently throughout this paper:
     \begin{itemize}
     	\item We define $P_\lambda$ as $\bigsqcup_{i = 1}^\ell P_i$, where each $P_i$ is a path with $\lambda_i$ vertices.
     	\item We define $ST_\lambda$ as $\bigsqcup_{i = 1}^\ell ST_i$, where each $ST_i$ is a star with $\lambda_i$ vertices.
     	\item We define $K_\lambda$ as a complete $\ell$-partite graph where the $i$-th part has $\lambda_i$ vertices.
    \end{itemize}
\end{definition}


\section{Steps, Routes and Marchs} \label{section:main-1}

We start this section by defining our main operation on graphs and will prove some results. Using them, in Subsection \ref{section:main-1-1}, we first prove the chromatic-basis theorem shown by Cho and van Willigenburg~\cite{cho2015chromatic}. Then, we prove an interesting version of the DNC-relation and will show that the algorithm provided in~\cite{aliste2023marked} can be viewed as an interesting case of our method. In Subsection \ref{section:main-1-2}, we prove one of the main results of this paper, which gives us the equivalency of the CSF and the $U$-polynomial on trees as a result.

\begin{definition}
	Let $G_1$ be a simple graph where $v_1, v_2$ and $v_3$ are three of its vertices such that $v_1, v_2$ and $v_1, v_3$ are adjacent, but $v_2, v_3$ are not. We achieve three new graphs by modifying $G_1$:
	\begin{itemize}
		\item[\texttt{mod1}] \label{mod1} We obtain graph $G_2$ by removing the edge $v_1, v_3$ and adding the edge $v_2, v_3$.
		\item[\texttt{mod2}] \label{mod2} We obtain graph $P_1$  by removing the edge $v_1, v_2$.
		\item[\texttt{mod3}] \label{mod3} We obtain graph $N_1$ by removing the edges $v_1, v_2$ and $v_1, v_3$, and adding the edge $v_2, v_3$.
	\end{itemize}
	Using these modifications and the resulting graphs, we define the following concepts:
	\begin{enumerate}[label=\alph*)]
		\item Define the \emph{step}$(G_1\to G_2)$ to be the ordered pair $(P_1, N_1)$.
		\item $P_1$ and $N_1$ are called the \emph{positive remainder} and the \emph{negative remainder} of the march respectively. Also, the set $Re=\left\{P_1, N_1\right\}$ is called the \emph{remainder set} of the step. Note that remainders have one less edge than $G_1$ and $G_2$.
		\item The graph $G_1$ is called \emph{stepable} to $G_2$, if there exist three vertices in $G_1$ with the above condition such that performing \texttt{mod1} on these vertices results in $G_2$. Note that this relation is symmetric, meaning that if  $G_1$ is stepable to $G_2$, then so is $G_2$ to $G_1$.
		\item We call a sequence of graphs $R=(G_1, G_2, \ldots, G_{k-1}, G_k)$ a \emph{route} from $G_1$ to $G_k$, if for each $1\leq i<k$, the graph $G_i$ is stepable to $G_{i + 1}$. Note that if $R=(G_1, G_2, \ldots, G_{k-1}, G_k)$ is a route from $G_1$ to $G_k$, then $R^{-1}=(G_k, G_{k-1}, \ldots, G_{2}, G_1)$ is a route from $G_k$ to $G_1$.
		\item Given two routes $R_1=(G_1, G_2, \ldots, G_{k-1}, G_k)$ and $R_2=(H_1, H_2, \ldots, H_{t-1}, H_t)$ such that $G_k = H_1$, we can concatenate these two and achieve $R_1 \concat R_2=(G_1, G_2, \ldots, G_{k-1}, G_k=H_1, H_2, \ldots, H_{t-1}, H_t)$, which is a route from $G_1$ to $H_t$.
		\item Let $R=(G_1, G_2, \ldots, G_{k-1}, G_k)$ be a route from a graph $G_1$ to a graph $G_k$, and for each $1\leq i<k$, we have $\text{step}(G_i\to G_{i+1})=(P_i,N_i)$. A \emph{march}$(G_1\to G_k)$ with respect to $R$  is an ordered pair $(P,N)$ such that $P=(P_1, P_2,\ldots, P_{k-1})$ and $N=(N_1, N_2,\ldots, N_{k-1})$.
		\item The sequences $P$ and $N$ are called the \emph{positive remainders} and the \emph{negative remainders} of the mentioned march respectively.
	\end{enumerate}
\end{definition}

Note that by Corollary \ref{cor:triless-rec-formula} for a $\text{step}(G_1\to G_2)=(P_1,N_1)$ we have:
\begin{equation} \label{step-form}
	X_{G_1} = X_{G_2} + X_{P_1} -  X_{N_1}.
\end{equation}
Now, consider a $\text{march}(G_1 \to G_k) = (P, N)$ along a route $R = (G_1, G_2, \ldots, G_{k-1}, G_k)$, where for each $1 \leq i < k$, we have $\text{step}(G_i \to G_{i+1}) = (P_i, N_i)$. Applying step formula (Equation \eqref{step-form}) iteratively, we obtain:
\begin{equation} \label{march-form}
	X_{G_1} = X_{G_k} + \sum_{i=1}^{k-1}X_{P_i} -  \sum_{i=1}^{k-1}X_{N_i}.
\end{equation}

Equation \eqref{march-form} is the core relation we will use in this section. It has two essential properties. Firstly, note that the only graph with the same number of edges as $G_1$ is $X_{G_k}$. All other graphs have precisely one less edge. Secondly, the two sigmas on the right-hand side of the equation have the same number of graphs, and there is a bijection between them.

Our first goal in this section is to show that for forests $F_1$ and $F_k$, if one starts from $F_1$ and performs the step function enough times in a proper route, then they will obtain $F_k$ as a result if and only if $\text{Part}(F_1)=\text{Part}(F_k)$. To prove this, we use the path graphs as a go-between. In the following two lemmas, we show that if the above condition is satisfied, one can find a suitable route from a forest $F$ to its corresponding path graph.

\begin{lemma} \label{lem:tree-path}
	Given a tree $T_1$ with $n$ vertices, there is a route $R=(T_1, T_2, \ldots, T_k)$ from $T_1$ to $T_k$, where $T_k$ is the path graph $P_n$.
\end{lemma}

\begin{proof}
    We proceed by induction on the longest path length in $T_1$. For the base case, consider $P_n$, which has the longest possible path length. In this case, let $T_2$ be one of the graphs that are stepable from $P_n$; then the route $R = (T_1 = P_n, T_2, T_3 = P_n)$ is a valid route from $P_n$ to $P_n$.

    Now, assume that $T_1$ is not a path graph. Let $L$ be one of its longest paths, and let $v_1$ be a vertex on $L$ with $\deg(v_1) > 2$. Such a vertex must exist, as $T_1$ is not a path graph. Let $v_2$ and $v_3$ be vertices adjacent to $v_1$ such that $v_2$ lies on $L$ while $v_3$ does not. Applying \texttt{mod1} produces a new graph $T_2$ in which $L$ has been extended to a longer path $L^{\prime}$, now the longest path in $T_2$.

    By the induction hypothesis, a route $R^{\prime} = (T_2, T_3, \ldots, T_k = P_n)$ exists from $T_2$ to $P_n$. Since $T_1$ is stepable to $T_2$, we can extend $R^{\prime}$ to form $R = (T_1, T_2, \ldots, T_k = P_n)$, establishing a valid route from $T_1$ to $P_n$.
\end{proof}

\begin{lemma} \label{lem:forest-path}
	Given a forest $F_1$ with $\text{Part}(F_1)=\lambda$, there is a route $R=(F_1, F_2, \ldots, F_k)$ from $F_1$ to $F_k$, where $F_k$ is the graph $P_\lambda$.
\end{lemma}

\begin{proof}
    To construct the route, assume $F_1 \neq P_\lambda$. Let $T_1, T_2, \ldots, T_t$ represent the connected components of $F_1$ that are not paths, and consider $T_1$, a tree with $q$ vertices. By Lemma \ref{lem:tree-path}, there exists a route $R{\prime} = (H_1, H_2, \ldots, H_t)$ from $H_1 = T_1$ to $H_t = P_q$. This route allows us to construct $R_1 = (F_1, F_2, \ldots, F_{t_1}, F_t)$, where each $F_i$ is derived from $F_1$ by replacing $T_1$ with $H_i$. Since each $H_i$ is stepable to $H_{i+1}$, so is $F_i$ to $F_{i+1}$, making $R_1$ a route from  $F_1$ to $H_t$. As $H_t$ includes one additional path component compared to $F_1$, we can repeat this process to generate routes $R_1, R_2, \ldots, R_{t-1}$, where for each $1 \leq i \leq t-2$, route $R_i$ is from $K_i$ to $K_{i+1}$, with $K_1 = F_1$ and $K_t = P_\lambda$. Notably, $K_{i+1}$ has one additional path component compared to $K_i$ due to the method we used. Concatenating these routes yields $R = R_1 \concat R_2 \concat \cdots \concat R_{t-2} \concat R_{t-1}$, forming a route from $K_1 = F_1$ to $K_t = P_\lambda$, as desired.
\end{proof}

In the preceding lemmas, stars could replace paths using a different routing method. Similarly to the previous approach, which increased the length of the longest path in the graph or a component, this method adjusts the maximum degree of the graph or a component, transforming them into star graphs. Consequently, we can derive a route from the initial graph to $ST_\lambda$. This is particularly interesting since by the end of this section it can be used as a routing method when you want to find the CSF in the chromatic-basis made by star graphs. This is interesting since it must give us the same values as the algorithm introduced in~\cite{aliste2023marked}, but it uses a different approach.

Now, we prove our desired statement by using the path graphs as a go-between.

\begin{proposition} \label{prop:forests-route}
	For any two forests $F_1$ and $F_k$, there exists a route $R=(F_1, F_2, \ldots, F_k)$ from $F_1$ to $F_k$ if and only if $\text{Part}(F_1)=\text{Part}(F_k)$.
\end{proposition}

\begin{proof}
    First, observe that since each step between graphs preserves both the size and the number of connected components, the partition function of the graph remains unchanged. Thus, two graphs with different partition functions cannot have a route. For the converse, by Lemma \ref{lem:forest-path}, there exist routes $R_1 = (H_1 = F_1, H_2, \ldots, H_t = P_\lambda)$ and $R_2 = (K_1 = F_k, K_2, \ldots, K_s = P_\lambda)$, where $\lambda = \text{Part}(F_1) = \text{Part}(F_k)$. Hence, $R_1 \concat R_2^{-1} = (H_1 = F_1, H_2, \ldots, H_t = P_\lambda = K_s, \ldots, K_2, K_1 = F_k)$ is a route from $ F_1$ to $ F_k$.
\end{proof}

From this point forward in this section, our primary approach is to relativize the CSF of a given graph to other graphs using the above proposition and Equation \eqref{march-form}.

In the first subsection, as applications, we use the above tools to prove two already-known facts using combinatorial methods. The first is \cite[Theorem 5]{cho2015chromatic}, which is the theorem that first introduced the idea of a chromatic-basis, and the second is the algorithm proposed by~\cite{aliste2023marked} for calculating the CSF in the star-basis using DNC-relation. We will show that this algorithm can be viewed as an interesting case of our method

In the second subsection, using the same method, we prove one of our main theorems, and from this theorem, we will infer an already known fact, which is $U$-polynomial and the CSF are equivalent for forests.

\subsection{Chromatic-bases and the DNC-Relation} \label{section:main-1-1}

\begin{definition} \label{def:forest-basis}
    Given an integer $n$, we assign a graph $G_\lambda$ to each $\lambda \vdash n$ such that $\text{Part}(G_\lambda)=\lambda$. We define a \emph{chromatic-basis} with respect to this assignment to be:
    $$
        B \triangleq \left\{ X_{G_\lambda} \mid \lambda \vdash n\right\}.
    $$
    A \emph{forest-basis} is a chromatic-basis with all the graphs $G_\lambda$ being forests.
\end{definition}

We aim to prove that a forest-basis is a $\mathbb{Q}$-basis of $\Lambda^n$. We first prove that a forest-basis $B$ is a generator set for $\mathcal{V}(\mathcal{G}_n)$. Then, by proving that the dimension of $\mathcal{V}(\mathcal{G}_n)$ equals $p(n)$, we conclude that $\mathcal{V}(\mathcal{G}_n) = \Lambda^n$. Therefore, show that $B$ is a $\mathbb{Q}$-basis of $\Lambda^n$.

Our approach is to use induction. To prove the statement for a general graph, we first prove it for forests. This is relatively straightforward due to the properties of Equation \eqref{march-form}. After that, in \ref{prop:cycle-to-girth3}, we show that every non-forest graph has a route to a 'good' graph, and by 'good' graph, we mean a graph containing a cycle of length $3$. The 'good' graph allows us to use the relation of Theorem \ref{thm:tri-rec-formula} and turn the 'good' graph into three other graphs with fewer edges. From that point, the induction will do the rest.

\begin{lemma} \label{lem:forest-forest-gen}
    A forest-basis $B$ is a generator set for $\mathcal{V}(\mathcal{F}_n)$.
\end{lemma}

\begin{proof}
    Let $F$ be a forest with partition $\text{Part}(F) = \lambda$, where $\lambda \vdash n$. Consider the set $B$ to be an arbitrary forest-basis $\left\{X_{F_{\mu}} \mid \mu \vdash n\right\}$. We will prove the lemma by induction on the number of edges of the graph. If the graph has no edges, then it corresponds to the forest $F_{(1,1,\ldots,1)}$, which always is a member of $B$. Hence, the base case holds. For the inductive step, assume that for any forest with fewer edges than $F$, the forest lies in the span of the forest-basis $B$. Now, consider the forest $F_\lambda \in B$ with partition $\text{Part}(F_\lambda) = \lambda$. By Proposition \ref{prop:forests-route}, there exists a route $R = (F_1 = F, F_2, \ldots, F_k = F_\lambda)$ from $F$ to $F_\lambda$. This route can be described as $\text{march}(F \to F_\lambda) = (P, N)$, and using the march formula (Equation \eqref{march-form}), we have that:
    $$
        X_F = X_{F_\lambda} + \sum_{i=1}^{k-1} X_{P_i} - \sum_{i=1}^{k-1} X_{N_i},
    $$
    where $P_i$s and $N_i$s are forests with fewer edges than $F$. By the induction hypothesis, since $P_i$s and $N_i$s have fewer edges than $F$, they lie within the span of $B$. Additionally, $F_\lambda$ is itself an element of $B$. Therefore, we conclude that $F$ is generated by the forest-basis $B$.
\end{proof}

\begin{proposition} \label{prop:cycle-to-girth3}
    Let $G$ be a graph that contains a cycle. Then there exists a graph $G^{\prime}$ with girth, i.e. the length of its shortest cycle, equal to $3$, and a route from $G$ to $G^{\prime}$.
\end{proposition}

\begin{proof}
    Let $C = v_1, v_2, \ldots, v_g$ be a cycle of length $g > 3$ in $G$. For each $1 \leq i \leq g-2$, define $G_i$ as the graph obtained from $G$ by removing the edge $v_2 v_3$ and adding the edge $v_2 v_{i+2}$. Let us denote $G^\prime$ as the last of these graphs that is, $G^\prime = G_{g-2}$. Observe that for each $1 \leq i \leq g-3$, the graph $G_i$ is stepable to $G_{i+1}$, which implies that the sequence $R = (G_1 = G, G_2, \ldots, G_{g-2} = G^\prime)$ forms a route from $G$ to $G^\prime$. Since $G^\prime$'s girth equals $3$, this concludes the proof.
\end{proof}

\begin{lemma} \label{lem:forest-graph-gen}
	A forest-basis $B$ is a generator set for $\mathcal{V}(\mathcal{G}_n)$.
\end{lemma}

\begin{proof}
    Let $G$ be a graph with partition $\text{Part}(G) = \lambda$, where $\lambda \vdash n$. If $G$ is a forest, the result follows directly from Lemma \ref{lem:forest-forest-gen}. Therefore, it suffices to prove the lemma for the case where $G$ contains at least one cycle.

    Just as with the previous lemma, this lemma will be proven by induction on the number of edges of the graph. Let $g$ denote the girth of $G$. If $g = 3$, then by Theorem \ref{thm:tri-rec-formula}, we can express $G$ in terms of graphs $G_{2,3}$, $G_{1,3}$, and $G_3$, each having fewer edges than $G$, such that:
    $$
        X_G = X_{G_{2,3}} + X_{G_{1,3}} - X_{G_3}.
    $$
    By the induction hypothesis, the graphs $G_{2,3}$, $G_{1,3}$, and $G_3$ are in the span of the forest-basis $B$. Thus, $G$ is also generated by $B$.
    
    Now, consider the case where $g > 3$. Using the result from Proposition \ref{prop:cycle-to-girth3}, there is a route from $G$ to a graph $G^\prime$ with a girth equal to $3$. This route can be described as $\text{march}(G \to G^\prime) = (P, N)$, and by using the march formula (Equation \eqref{march-form}), we obtain the following relation:
    $$
        X_G = X_{G^\prime} + \sum_{i=1}^{k-1} X_{P_i} - \sum_{i=1}^{k-1} X_{N_i}.
    $$
    By a similar argument to that used in the proof of Lemma \ref{lem:forest-forest-gen}, each $P_i$ and $N_i$ lies in the span of the forest-basis $B$. Moreover, since $G^\prime$ is a graph with girth $3$, it can also be generated by $B$, as shown earlier. Combining these results, we conclude that $G$ can be generated by the forest-basis $B$.
\end{proof}

\begin{lemma} \label{lem:dim-vec-set-graph}
    The dimension of the $\mathcal{V}(\mathcal{G}_n)$ equals $p(n)$.
\end{lemma}

We postpone the proof of this lemma until the end of Section \ref{section:main-2}. There, we will provide a family of graphs whose CSFs are independent.

\begin{definition} \label{def:g-lambda}
    Let $\left\{G_k\right\}_{k \geq 1}$ be a set of connected graphs such that $G_k$ has $k$ vertices for each $k$, and let $\lambda = (\lambda_1, \ldots ,\lambda_\ell)$ be a partition. Then,
    $$
        G_\lambda \triangleq G_{\lambda_{1}} \bigsqcup \cdots \bigsqcup G_{\lambda_{\ell}}.
    $$
    That is, $G_\lambda$ is the graph whose connected components are $G_{\lambda_{1}}, \ldots, G_{\lambda_{\ell}}$.
\end{definition}

The following theorem is a weaker version of the theorem shown in~\cite{cho2015chromatic} since we are only proving the theorem for the case forests. Although our statement is weaker, but our proof is combinatorial and shows how these bases actually span to generate the space, while the proof provided in~\cite{cho2015chromatic} is algebraic and based on the independence of graphs CSFs.

\begin{theorem} \label{thm:vec-set-forest-base}
    Let $\left\{T_k\right\}_{k \geq 1}$ be a set of trees such that $T_k$ has $k$ vertices for each $k$. Then,
    $$
        \left\{ X_{T_\lambda} \mid \lambda \vdash n \right\}
    $$ 
    is a $\mathbb{Q}$-basis of $\Lambda^n$. Plus, we have that:
    $$
        \Lambda = \mathbb{Q}[X_{T_1}, X_{T_2}, \ldots],
    $$
    and the $X_{T_k}$s are algebraically independent over $\mathbb{Q}$.
\end{theorem}

\begin{proof}
    According to the Lemmas \ref{lem:forest-graph-gen} and \ref{lem:dim-vec-set-graph}, the set $ \left\{ X_{T_\lambda} \mid \lambda \vdash n \right\} $ is a $\mathbb{Q}$-basis for $\Lambda^n$. Therefore, since $X_{T_\lambda} = X_{T_{\lambda_{1}}}\cdot X_{T_{\lambda_{2}}} \cdot \cdots \cdot X_{T_{\lambda_{\ell}}}$, where $\lambda = (\lambda_1, \ldots, \lambda_\ell)$, we can deduct that:
    $$
        \Lambda = \mathbb{Q}[X_{T_1}, X_{T_2}, \ldots].
    $$
    Therefore, the $X_{T_k}$s are algebraically independent over $\mathbb{Q}$.
\end{proof}

The following definition and the two corollaries after that are handy tools that we will exploit in the next subsection.

\begin{definition} \label{def:int-coef}
    Let $B$ be a chromatic-basis defined in Definition \ref{def:forest-basis}. For a graph $G$ with $n$ vertices, we define $X_{G, B}$ as the \emph{CSF of $G$ expressed in the chromatic-basis $B$}. For $\lambda \vdash n$, we denote by $[\lambda] X_{G, B}$ the coefficient of $X_{G, B}$ with respect to $X_{G_\lambda} \in B$.
\end{definition}

The following corollaries are particularly important since they actually show how a forest $F$ with $\text{Part}(F)=\lambda$ can be a representative of $\lambda$ inside a chromatic-basis. Note that the forest-basis $B$ in these corollaries is unrelated to $F$ itself. So, not only $\lambda$ can represent a non-zero coefficient, but it represents a coefficient equal to one. In the next subsection, we will discuss the coefficients of the CSF shown in a forest-basis more thoroughly.

\begin{corollary} \label{cor:coef-zero}
    Let $G$ be a graph with $n$ vertices such that $\text{Part}(G)=\lambda$. If $[\mu]X_{G, B}\neq 0$, then $\mu\leq \lambda$.
\end{corollary}

\begin{proof}
    We use induction on the number of edges of the graph. For the base case, consider an edgeless graph. We examine two cases:
    \begin{itemize}
        \item \textbf{$G$ is a forest:} In this case, using Proposition \ref{prop:forests-route} there exists a route from $G$ to $F_\lambda$, where $X_{F_\lambda}\in B$. Using the marching formula (Equation \eqref{march-form}) on $R$, we get:
        $$
            X_G = X_{F_\lambda} + \sum_{i=1}^{k-1}X_{P_i} -   \sum_{i=1}^{k-1}X_{N_i}.
        $$
        Firstly, note that $P_i$s and $N_i$s have fewer edges than $G$. Also, for all $1 \leq i \leq k-1$, we have $\text{Part}(P_i)\leq \lambda$ and $\text{Part}(N_i)\leq \lambda$. This implies that if $\mu \vdash n$ such that $[\mu]X_{G, B}\neq 0$ and $\mu \nleq \lambda$, then by the induction hypothesis, $[\mu]X_{P_i, B} = 0$ and $[\mu]X_{N_i, B} = 0$. Since $[\mu]X_{F_\lambda, B}\neq 0$ only if $\lambda = \mu$, we achieve the desired result.

        \item \textbf{$G$ has a cycle:} In this case, by Proposition \ref{prop:cycle-to-girth3}, there exist a graph $G^\prime$ and a route from $G$ to $G^\prime$ such that $G^\prime$ contains a triangle. By Theorem \ref{thm:tri-rec-formula}, there exist graphs $G_1$, $G_2$ and $G_3$ such that $X_{G^\prime} = X_{G_1} + X_{G_2} - X_{G_3}$. Applying Equation \eqref{march-form}, we have:
        $$
            X_{G,B} = X_{G^\prime, B} + \sum_{i=1}^{k-1}X_{P_i,B} -  \sum_{i=1}^{k-1}X_{N_i,B} = X_{G_1,B} + X_{G_2,B} - X_{G_3,B} + \sum_{i=1}^{k-1}X_{P_i,B} -  \sum_{i=1}^{k-1}X_{N_i,B}.
        $$
        Since all $P_i$s, $N_i$s, and $G_j$s have fewer edges than $G$, and we also have $\text{Part}(P_i)\leq \lambda$, $\text{Part}(N_i) \leq \lambda$ and $\text{Part}(G_j)\leq \lambda$, the result follows by the induction hypothesis.
    \end{itemize}
\end{proof}

\begin{corollary} \label{cor:coef-one}
    Let $F$ be a forest such that $\text{Part}(F) = \lambda$. Then $[\lambda]X_{F, B} = 1$.
\end{corollary}

\begin{proof}
    By Proposition \ref{prop:forests-route}, there exists a route $R=(F_1=F, F_2, \ldots, F_{k+1}=F_\lambda)$. Applying the marching formula (Equation \eqref{march-form}) to $R$, we obtain:
    $$
        X_F = X_{F_\lambda} + \sum_{i=1}^{k}X_{P_i} -  \sum_{i=1}^{k}X_{N_i}.
    $$
    Note that $\ell(\text{Part}(P_i))$ and $\ell(\text{Part}(N_i))$ are both equal to $\ell(\lambda) + 1$, for all $1 \leq i \leq k$. Therefore, by Corollary \ref{cor:coef-zero}, we conclude that $[\lambda] \left( \sum_{i=1}^{k} X_{P_i} - \sum_{i=1}^{k} X_{N_i} \right) = 0$. Thus, $[\lambda]X_{F, B} = [\lambda]X_{F_{\lambda, B}} = 1$, as desired.
\end{proof}

Now, for the last statement of this subsection, we prove an interesting version of the DNC-relation first shown in~\cite{aliste2023marked}. An important use of the DNC-relation is the results provided in~\cite{gonzalez2024chromatic} where they used the algorithm introduced in \cite{aliste2017trees}. This algorithm's objective is similar to what we have done so far in this section. It uses the relation to calculate the coefficients of the CSF of a graph in the star-basis, which is the chromatic-basis made by $ST_\lambda$s. In fact, this algorithm is so similar to our work that it can be viewed as an instance of it. In the following proposition, we show that the relation, in the case of the forests, can be viewed as a route. This means that the relation is actually a method of making routes. Routes that can connect any forest to the correlated $ST_\lambda$. So, in the case of forests, the mentioned algorithm is actually the method we introduced, but with a specified routing algorithm.

\begin{proposition} \label{prop:csf-dnc-relation}
    Let $G_1$ be a graph, and let $wu$ be one of its edges. Let $v_1, \ldots, v_k$ be the vertices adjacent to $u$ but not to $w$. Define $G_2$ as the graph obtained from $G_1$ by removing each edge $uv_i$ and adding the edge $wv_i$ for each $1 \leq i \leq k$. If we define $P^{\prime} = G_1 - \left\{wu\right\}$ and $N^{\prime} = G_2 - \left\{wu\right\}$, then:
    $$
        X_{G_1} = X_{G_2} + X_{P^\prime} - X_{N^\prime}.
    $$
\end{proposition}

\begin{proof}
    For each $0 \leq i \leq k$, define the graph $H_i$ as the graph obtained from $G_1$ by removing the edges $uv_1, \ldots, uv_i$ and adding the edges $wv_1, \ldots, wv_i$. Note that by definition, $H_0 = G_1$ and $H_k = G_2$. Since $H_i$ is stepable to $H_{i+1}$ for each $0 \leq i \leq k - 1$, the sequence $R = (H_0 = G_1, H_1, \ldots, H_k = G_2)$ forms a route from $G_1$ to $G_2$. Describing this route as $\text{march}(G_1 \to G_2) = (P, N)$, the marching formula (Equation \eqref{march-form}) gives:
    $$
        X_{G_1} = X_{G_2} + \sum_{i = 0}^{k-1} X_{P_i} - \sum_{i = 0}^{k-1} X_{N_i}.
    $$
    Observing that $N_i = P_{i+1}$ for each $0 \leq i \leq k - 2$, and noting $P_0 = P^\prime$ and $N_{k-1} = N^\prime$, we conclude:
    $$
        X_{G_1} = X_{G_2} + X_{P^\prime} - X_{N^\prime},
    $$
    as required.
\end{proof}

As shown in the proof, the relation is equivalent to finding a route between $G_1$ and $G_2$, and by considering the march of this route and applying Equation \eqref{march-form} we will end up with the desired relation. Note that, in the case of forests, none of the neighbors of $u$ can be adjacent to $w$, except for $w$ itself. So, in this case, the above relation will be the same as the DNC-relation. Note that how this version is faster in the sense that it avoids a lot of calculations because of the cancellations. This brings up a new question. Where else can we make such cancellations to make the final relation simpler? 


\subsection{CSF, U-Polynomial and the Tree Conjecture} \label{section:main-1-2}

In this subsection, our aim is to prove Theorem \ref{thm:U-equiv}. This theorem shows, in the case of forests, how much the $U$-polynomial behaves like the CSF when viewed in the proper chromatic-basis. To prove the theorem, we first show that the relation of Equation \eqref{step-form} is still true if we replace the CSF with the $U$-polynomial. This is particularly important since our work in this section is almost entirely based on this relation. So, having this relation be true for $U$-polynomials means that we can import most of our tools from the CSF to the $U$-polynomial.

\begin{lemma} \label{lem:upoly-relation}
    Let $F_1$ be a forest where $v_1, v_2$, and $v_3$ are three of its vertices such that $v_1, v_2$ and $v_1, v_3$ are adjacent, but $v_2, v_3$ are not. Let $F_2$, $P_1$, and $N_1$ be the forests obtained by performing \texttt{mod1}, \texttt{mod2}, and \texttt{mod3} operations on $F_1$. Then we have:
    $$
        U_{F_1} = U_{F_2} + U_{P_1} - U_{N_1}.
    $$
\end{lemma}

\begin{proof}
    Let $e_1 = v_1 v_2$, $e_2 = v_1 v_3$, and $e_3 =  v_2 v_3$. For each $A \subseteq E(F_1)$, define $A^{\prime} \subseteq E(F_2)$ as follows:
    \begin{itemize}
        \item if $e_2 \not\in A$, then $A^\prime = A$,
        \item if $e_2 \in A$, then $A^\prime = \left(A - \left\{e_2\right\}\right) \cup \left\{e_3\right\}$.
    \end{itemize}
    Note that for each $A \subseteq E(F_1)$, there is exactly one corresponding $A^\prime \subseteq E(F_2)$.
    By definition, $U_F(x) = \sum_{A \subseteq E(F)} x_{\lambda (A)}$. For each $A$, we aim to enumerate and compare $x_{\lambda(A)}$ in $[\lambda(A)] U_{F_1}$ and $x_{\lambda(A^\prime)}$ in $[\lambda(A^\prime)] U_{F_2}$. We proceed by dividing the statement into two cases:

    \begin{itemize}
        \item \textbf{$e_1 \in A$:} If $e_2 \not\in A$, then $e_3 \not\in A^\prime$. Therefore, $\text{Part}(F_1 |_{A}) = \text{Part}(F_2 |_{A^\prime})$, which implies $x_{\lambda(A)} = x_{\lambda(A^\prime)}$. If $e_2 \in A$, then $e_3 \in A^\prime$. In this situation, we observe that in both $\text{Part}(A)$ and $\text{Part}(A^\prime)$, the components that do not include vertices $v_1$, $v_2$, and $v_3$ are identical. Additionally, the component that includes any one of these vertices in $A$ will contain the other two vertices as well, since $v_3$ is connected to $v_1$ via $e_2$, and $v_1$ is connected to $v_2$ via $e_1$. This structure holds similarly for $A^\prime$. The rest of these two components are thus equivalent by the definition of \texttt{mod1}. Since these three vertices are connected by two edges in each of the components, we conclude that $x_{\lambda(A)} = x_{\lambda(A^\prime)}$. Therefore, we can write:
        $$
            \sum_{ e_1 \in A \subseteq F_1} x_{\lambda (A)} - \sum_{ e_1 \in A^\prime \subseteq F_2} x_{\lambda (A^\prime)} = 0.
        $$
        \item \textbf{$e_1 \not\in A$:} We know that $P_1 = G_1 - e_1$ and $N_1 = G_2 - e_1$. Therefore, we can say that $A$  is a subset of edges of $P_1$ and $A^\prime$ is a subset of edges of $N_1$, each containing the exact same vertices and edges as $F_1$ and $F_2$, respectively. In fact, the set of all subsets $A$ that do not contain $e_1$ makes up the set of all subsets of $P_1$, and similarly, the set of all subsets $A^\prime$ that do not contain $e_1$ makes up the set of all subsets of $N_1$. Using these facts, we can write:
        $$
            \sum_{ e_1 \not\in A \subseteq F_1} x_{\lambda (A)} = \sum_{A \subseteq P_1} x_{\lambda (A)} = U_{P_1}, \quad \text{and} \quad \sum_{ e_1 \not\in A^\prime \subseteq F_2} x_{\lambda (A^\prime)} = \sum_{A^\prime \subseteq N_1} x_{\lambda (A^\prime)} = U_{N_1}.
        $$
    \end{itemize}
    Combining these results, we get:
    $$
        U_{F_1} - U_{F_2} = U_{P_1} - U_{N_1},
    $$
    as desired.
\end{proof}

Now that we have established that the relation also holds for $U$-polynomials, we start proving the theorem. For a forest $F$, let us call the number $\ell (\text{Part}(F))$ its \emph{level}. The following lemma shows how the CSF and the $U$-polynomial in forest-basis act the same.

\begin{lemma} \label{lem:firs-lvl-eval}
    Let $F_1$ and $F_2$ be forests with $n$ vertices such that  $\text{Part}(F_1) = \text{Part}(F_2)$. Let $B$ be a forest-basis. Then, for all $\lambda \vdash n$, such that $\ell(\lambda) = \ell(\text{Part}(F_1)) + 1$, we have:
    $$
        [\lambda]X_{F_1, B} - [\lambda]X_{F_2, B} = [\lambda]U_{F_1} - [\lambda]U_{F_2}.
    $$
\end{lemma}

\begin{proof}
    By previous discussion, there exists a route $R$ from $F_1$ to $F_2$ with $k - 2$ forests between them, yielding:
    $$
        X_{F_1} = X_{F_2} + \sum_{i=1}^{k-1} X_{P_i} - \sum_{i=1}^{k-1} X_{N_i}.
    $$
    By Corollaries \ref{cor:coef-zero} and \ref{cor:coef-one}, we know that for each $1 \leq i \leq k-1$ , the value of $[\lambda]X_{P_i}$ is either zero or one, depending on whether $\text{Part}(P_i) = \lambda$. The same holds for $[\lambda]X_{N_i}$. Therefore, we can express:
    $$
        [\lambda]X_{F_1} - [\lambda]X_{F_2} = \#_{\lambda} P_i - \#_{\lambda} N_i,
    $$
    where $\#_{\lambda} P_i$ and $\#_{\lambda} N_i$ represent the number of $P_i$ and $N_i$ forests such that $\text{Part}(P_i) = \lambda$ and $\text{Part}(N_i) = \lambda$, respectively.
    On the other hand, with respect to the route $R$, we have:
    $$
        U_{F_1} = U_{F_2} + \sum_{i=1}^{k-1} U_{P_i} - \sum_{i=1}^{k-1} U_{N_i}.
    $$
    By the same reasoning, $[\lambda]U_{P_i}$ is also either zero or one, depending on whether $\text{Part}(P_i) = \lambda$. Thus, we get:
    $$
        [\lambda]U_{F_1} - [\lambda]U_{F_2} = \#_{\lambda} P_i - \#_{\lambda} N_i = [\lambda]X_{F_1} - [\lambda]X_{F_2},
    $$
    completing the proof.
\end{proof}

 Note that the relation shown in this lemma happens at exactly one level higher than the level of $F_1$ and $F_2$. Also, you might have noticed that we did not use the elements of basis $B$ in the proof, meaning that it does not matter how you choose the basis as long as it remains a forest-basis. You can even choose different bases for $F_1$ and $F_2$, and the lemma will still hold. All these notions point to the fact that the $U$-polynomial behaves almost the same as the CSF in the case of forests.

\begin{definition} \label{def:X_k}
    Let $G$ be a graph with $n$ vertices such that $\text{Part}(G)=\mu$, and let $B$ be a chromatic-basis such that $X_{G, B} (x) =  \sum_{\lambda \vdash n} c_\lambda X_{G_\lambda}$. For a $k \in \mathbb{Z}^{+} \cup \left\{0\right\}$, we define $X^{k}_{G, B} (x)$ to be:
    $$
        X^{k}_{G, B} (x) \triangleq \sum_{\substack{\lambda \vdash n \\ \ell(\lambda) \leq \ell (\mu) + k}} c_\lambda X_{G_\lambda}.
    $$
\end{definition}

A new question that comes to mind after Lemma \ref{lem:firs-lvl-eval} is how far can we go in terms of levels and still have the relation to be true? We try to answer this question in the following proposition. We will call this number the \emph{corner number} in Definition \ref{def:corner}.

\begin{proposition} \label{prop:lambda-zero-march}
    Let $F_1$ and $F_2$ be forests with $n$ vertices such that $\text{Part}(F_1) = \text{Part}(F_2) = \mu$. Let $B$ be a forest-basis. If for some $k \in \mathbb{Z}^{+} \cup \left\{0\right\}$ we have $X^{k}_{F_1, B} - X^{k}_{F_2, B} = 0$, then the following statements hold:
    \begin{itemize}
        \item There exist sets $\left\{ P_i : i \in \{1, \ldots, t\} \right\}$ and $\left\{ N_i : i \in \{1, \ldots, t\} \right\}$ such that $\ell(\text{Part}(P_i)) = \ell(\text{Part}(N_i)) = \ell(\mu) + k + 1$, and  
        $$
            X_{F_1} - X_{F_2} = \sum_{i=1}^{t} X_{P_i} - \sum_{i=1}^{t} X_{N_i} \quad \text{ and } \quad U_{F_1} - U_{F_2} = \sum_{i=1}^{t} U_{P_i} - \sum_{i=1}^{t} U_{N_i}.
        $$
        \item For all $\lambda \vdash n$ such that $\ell(\lambda) \leq \ell(\mu) + k + 1$, we have:
        $$
            [\lambda]X_{F_1, B} - [\lambda]X_{F_2, B} = [\lambda]U_{F_1} - [\lambda]U_{F_2}.
        $$
    \end{itemize}
\end{proposition}

\begin{proof}
    We prove both statements by induction on $k$.
    Lemma \ref{lem:firs-lvl-eval} acts as the basis of induction for $k=0$.
    Now, assume the statement holds for $k - 1$. Then, there exist sets of size $t$ containing $P_i$s and $N_i$s such that $\ell(\text{Part}(P_i)) = \ell(\text{Part}(N_i)) = \ell(\mu) + k$ and we have:
    $$
        X_{F_1} - X_{F_2} = \sum_{i=1}^{t} X_{P_i} - \sum_{i=1}^{t} X_{N_i}.
    $$
    Since $X_{F_1}^k - X_{F_2}^k = 0$, let $\lambda$ be a partition with length $\ell(\mu) + k$. Then,
    $$
        [\lambda]X_{F_1} - [\lambda]X_{F_2} = \sum_{i=1}^{t} [\lambda]X_{P_i} - \sum_{i=1}^{t} [\lambda]X_{N_i} = 0.
    $$
    By Corollaries \ref{cor:coef-zero} and \ref{cor:coef-one}, $[\lambda]X_{P_i}$ and $[\lambda]X_{N_i}$ are either zero or one, depending on whether $\text{Part}(P_i) = \lambda$ and $\text{Part}(N_i) = \lambda$ or not. Thus, the number of $P_i$s and $N_i$s with partition $\lambda$ is equal. This argument holds for each partition of length $\ell(\mu) + k$, allowing us to pair each $P_i$ with partition $\lambda^i$ with an $N_i$ of the same partition. Without loss of generality, we assume that $P_i$ is paired with $N_i$.
    Consequently, we can write:
    $$
        X_{F_1} - X_{F_2} = \sum_{i=1}^{t} (X_{P_i} - X_{N_i}).
    $$
    Since $P_i$ and $N_i$ are paired with the same partitions, there exists a route $R$ between $P_i$ and $N_i$, allowing us to express $X_{P_i} - X_{N_i}$ as:
    $$
        X_{P_i} - X_{N_i} = \sum_{j=1}^{t^{\prime}_i} X_{P^{\prime}_{ij}} - \sum_{j=1}^{t^{\prime}_i} X_{N^{\prime}_{ij}}.
    $$
    Substituting into our equation, we can write:
    $$
        X_{F_1} - X_{F_2} = \sum_{i=1}^{t} \left( \sum_{j = 1}^{t^\prime_i} X_{P^{\prime}_{ij}} - \sum_{j = 1}^{t^\prime_i}X_{N^{\prime}_{ij}} \right),
    $$
    where $\ell(\text{Part}(P_i)) = \ell(\text{Part}(N_i)) = \ell(\mu) + k + 1$. Thus, the statement holds for $k$.
    Furthermore, by Lemma \ref{lem:upoly-relation} and the fact that $[\lambda] U_F$ is either zero or one depending on whether $\text{Part}(F) = \lambda$ or not when $\ell (\lambda) = \ell (\text{Part}(F))$, the same argument applies.

    For the second statement, by Lemma \ref{lem:firs-lvl-eval}, the base case $k = 0$ is proved. Assume the statement holds for $k - 1$; we will prove it holds for $k$.
    If $\ell(\lambda) \leq \ell (\mu) + k$, the induction hypothesis holds and we are done. So suppose that $\ell(\lambda) = \ell (\mu) + k + 1$. In the previous step, we showed that if $X^{k}_{F_1} - X^{k}_{F_2} = 0$, there exist two sets containing $P_i$s and $N_i$s with $\ell(\text{Part}(P_i)) = \ell(\text{Part}(N_i)) = \ell(\mu) + k + 1$ such that:
    $$
        X_{F_1} - X_{F_2} = \sum_{i=1}^{t} X_{P_i} - \sum_{i=1}^{t} X_{N_i} \quad \text{ and } \quad U_{F_1} - U_{F_2} = \sum_{i=1}^{t} U_{P_i} - \sum_{i=1}^{t} U_{N_i}.
    $$
    For an arbitrary partition $\lambda$ with $\ell(\lambda) = \ell(\mu) + k + 1$, we can write:
    $$
        [\lambda]X_{F_1} - [\lambda]X_{F_2} = \sum_{i=1}^{t} [\lambda]X_{P_i} - \sum_{i=1}^{t} [\lambda]X_{N_i}.
    $$
    By Corollaries \ref{cor:coef-zero} and \ref{cor:coef-one}, $[\lambda]X_{P_i}$ and $[\lambda]X_{N_i}$ are either zero or one depending on whether their partition equals $\lambda$. Thus, we can write:
    $$
        [\lambda]X_{F_1} - [\lambda]X_{F_2} = \#_{\lambda} P_i - \#_{\lambda} N_i,
    $$
    where $\#_{\lambda} P_i$ and $\#_{\lambda} N_i$ denote the number of $P_i$s and $N_i$s with $\text{Part}(P_i) = \lambda$ and $\text{Part}(N_i) = \lambda$, respectively.
    The same holds for $U$-polynomials, as we have:
    $$
        [\lambda]U_{F_1} - [\lambda]U_{F_2} = \sum_{i=1}^{t} [\lambda]U_{P_i} - \sum_{i=1}^{t} [\lambda]U_{N_i},
    $$
    and $[\lambda]U_{P_i}$ and $[\lambda]U_{N_i}$ are either zero or one depending on whether their partition equals $\lambda$. Hence, we conclude that:
    $$
        [\lambda]U_{F_1} - [\lambda]U_{F_2} = \#_{\lambda} P_i - \#_{\lambda} N_i = [\lambda]X_{F_1} - [\lambda]X_{F_2}.
    $$
    This completes the proof.
\end{proof}

Again, note that the choice of $B$ does not matter and can even be different for $F_1$ and $F_2$. Now, we define the \emph{corner number} formally.

\begin{definition} \label{def:corner}
    Let $F$ be a forest with $n$ vertices such that $\text{Part}(F) = \lambda$, and let $B$ be an arbitrary chromatic-basis such that $X_{F^\prime}$ where $F^\prime$ is a forest and $\text{Part}(F^\prime) = \lambda$. We say $k$ is $F$'s \emph{corner number} with respect to $B$ if $X^{k-1}_{F,B} = X_{F^\prime}$ but $X^k_{G,B} \neq X_{F^\prime}$.
\end{definition}

Note that since for $k = 0$, we have $X^k_{G,B} = X_{F^\prime}$. Thus, the corner number of forest $F$ with respect to a forest-basis $B$ is at least $1$. Now, as a result of the proposition, we achieve our desired theorem.

\begin{theorem} \label{thm:U-equiv}
    Let $F_1$ and $F_2$ be forests with $n$ vertices such that $\text{Part}(F_1) = \text{Part}(F_2) = \mu$. Let $B$ be a forest-basis such that $X_{F_2} \in B$. Then, for all $\lambda \vdash n \text{ such that } \ell(\lambda) \leq k$, and $\lambda \neq \mu$, the following holds:
    $$
        [\lambda]X_{F_1, B} = [\lambda]U_{F_1} - [\lambda]U_{F_2},
    $$
    where $k$ is the corner number of $F_1$ with respect to $B$. Also, for $\lambda = \mu$ we have $[\lambda]X_{F_1, B} = 1$ and $[\lambda]U_{F_1} - [\lambda]U_{F_2} = 0$.
\end{theorem}

\begin{proof}
    The theorem follows directly from Proposition \ref{prop:lambda-zero-march}.
\end{proof}

The most important aspect of this theorem is how we proved it. Note that although we have restricted the forest-basis $B$ by putting $X_{F_2}$ inside, the rest of $B$ is still entirely arbitrary. Furthermore, by proving Lemma \ref{lem:upoly-relation} and Proposition \ref{prop:lambda-zero-march}, we showed that one way to view these polynomials is by defining them only using a simple relation, meaning that you can define problems like the tree conjecture without defining a meaningful function, such as the CSF or the $U$-polynomial. In this section, we have mostly avoided what the CSF and the $U$-polynomial actually are and instead try to relativize graphs using their properties. One could probably argue that the main idea of this paper is to relate graphs to each other.

Finally, we can deduce a previously achieved result in~\cite{noble1999weighted} from our theorem.

\begin{corollary}
    For two forests $F_1$ and $F_2$, we have $X_{F_1} = X_{F_2}$ if and only if $U_{F_1} = U_{F_2}$.
\end{corollary}

\begin{proof}
    The corollary follows directly from Theorem \ref{thm:U-equiv}.
\end{proof}

Another paper worth mentioning is the work of Chan and Crew~\cite{chan2024graph}. Although their approach is different from ours, some of their results are achievable through ours. For instance, Theorem 3.1 or Lemma 3.3 can be inferred from Equation \eqref{march-form}. 


\section{Homomorphisms and Calculating the CSF} \label{section:main-2}

The primary source of inspiration for this section was the reconstruction conjecture. Since this conjecture is known to be true for the case of forests~\cite{harary2006survey}, a reasonable approach for tree conjecture would be to check whether one can reach the deck of cards of a forest only having its CSF. To walk down this road, one needs to find a connection between the induced subgraphs with one less vertex and the CSF of the main tree under study. Thus, in this section, we develop a linear relation that can give us most of the CSF of our main graph by having these induced subgraphs and their CSF. In this section, our basis of interest for the CSF is the $m_\lambda$-basis.

\begin{definition}
    Given a graph $G$ with $n$ vertices and a partition $\lambda = (\lambda_{1}, \lambda_{2}, \ldots, \lambda_{\ell}) \vdash n$, we say a set
    $$
       P = \left\{ \{v_1, v_2, \ldots, v_{\lambda_{1}}\}, \{v_{\lambda_{1}+1}, \ldots, v_{\lambda_{1} + \lambda_{2}}\}, \ldots, \{v_{\lambda_1 + \cdots + \lambda_{\ell - 1} + 1}, \ldots, v_{n}\} \right\}
    $$
    is an \emph{independent $\lambda$-partition} of $G$ if, for every $P \in A$, there is no edge $e$ between the vertices in $P$. Each set $P \in A$ is called a \emph{part} of $A$. Moreover, $P$ is a \emph{$k$-part} of $A$ if it contains $k$ elements.
\end{definition}

\begin{definition}
    Given two graphs $H$ and $G$ with an isomorphism $\phi : H \to G$, for an independent $\lambda$-partition $P = \left\{ Q_1, Q_2, \ldots, Q_\ell \right\}$ of $H$, we define $\phi(P)$ as:
    $$
       \phi(P) \triangleq \left\{ \phi(Q_1), \phi(Q_2), \ldots, \phi(Q_\ell) \right\},
    $$
    where $\phi(Q_i) = \left\{ \phi(v) \mid v \in Q_i \right\}$.
    Note that $\phi(P)$ is itself an independent $\lambda$-partition of $G$.
\end{definition}

The idea behind Theorem \ref{prop:lambda-mat-vec-full-homo} is to see induced subgraphs of our graph $G$ as separate smaller graphs embedded in $G$ by a morphism. We can use this view to induce $G$'s independent partitions onto smaller graphs and vice versa, resulting in the desired relation.

\begin{theorem} \label{prop:lambda-mat-vec-full-homo}
    For a given graph $G$ with $n$ vertices and a partition $\lambda^1 = (\lambda_{1}, \lambda_{2}, \ldots, \lambda_{\ell}) \vdash n$, we have:
    $$
       c^{G}_{\lambda^1} = \binom{n-m}{k-m}^{-1} \sum_{H : |V(H)| = k} c^{H}_{\lambda^2} \binom{G}{H},
    $$
    where $\lambda^{1} \sim \lambda^{2}$, $\lambda^{1 \ast} \vdash m$, and $\lambda^{2} \vdash k$.
\end{theorem}

\begin{proof}
    We aim to enumerate all independent $\lambda^{1}$-partitions of $G$. First, we introduce a method to enumerate these sets using independent $\lambda^{2}$-partitions of all $k$-vertex subgraphs of $G$. Second, we show that each independent $\lambda^{1}$-partition of $G$ is counted at least once using this method. Lastly, we prove that each independent $\lambda^{1}$-partition of $G$ is enumerated exactly $\binom{n-m}{k-m}$ times; thus, by dividing the total count $\sum_{H : |V(H)| = k} c_{\lambda^{2}}^{H} \binom{G}{H}$ by this factor, we obtain $c_{\lambda^{1}}^{G}$.

    Let $H$ be a graph with $k$ vertices. Note that if there is no subgraph $K$ of $G$ isomorphic to $H$, then $\binom{G}{H} = 0$. For each graph $H$ with $k$ vertices, and each subgraph of $G$ isomorphic to $H$, let $\phi_{H}^{K} : H \rightarrow K$ be a fixed isomorphism from $H$ to $K$. Let $P_1, P_2, \ldots, P_{c^{H}_{\lambda^{2}}}$ be the $c^{H}_{\lambda^{2}}$ distinct independent $\lambda^{2}$-partitions of $H$. Define $W = \left\{ \{v_i\} \mid v_i \in (G - K) \right\}$. Then, for each $P_j$, let $P^{\prime}_{j}$ be the union of $\phi_{H}^{K}(P_j)$ with $W$. Since $\lambda^{1} \sim \lambda^{2}$, each $P^{\prime}_{j}$ is an independent $\lambda^{1}$-partition of $G$. For each $H$, we apply the $\binom{G}{H}$ fixed isomorphisms from $H$ to distinct subgraphs of $G$, and for each $H$ there are $c_{\lambda^{2}}^{H}$ independent $\lambda^{2}$-partitions. Thus, we generate $\sum_{H : |V(H)| = k} c_{\lambda^{2}}^{H} \binom{G}{H}$ independent $\lambda^{1}$-partitions in $G$; however, these partitions are not necessarily distinctive.

    We now claim that the described method enumerates each independent $\lambda^{1}$-partition of $G$ at least once. Let $P^{\prime}_{k}$ be an independent $\lambda^{1}$-partition of $G$. Since $P^{\prime}_{k}$ contains $n-m$ number of $1$-parts, let $W$ represent all $1$-parts of $P^{\prime}_{k}$. By selecting $k-m$ elements from $W$ along with $P^{\prime}_{k} - W$, we obtain a $k$-vertex subgraph of $G$, denoted as $K$. Let $H^{\prime}$ be a graph isomorphic to $K$. Since $H^{\prime}$ contains $k$ vertices, it has already been counted in $\sum_{H : |V(H)| = k} c_{\lambda^{2}}^{H} \binom{G}{H}$. Therefore, $P^{\prime}_{k}$ has been enumerated at least once.

    Finally, we determine how many times each independent $\lambda^{1}$-partition has been enumerated. Let $P^{\prime}_{k}$ be an independent $\lambda^{1}$-partition of $G$. As noted, $P^{\prime}_{k}$ has $n-m$ $1$-parts; we choose $k-m$ of these $1$-parts and exclude the others to obtain a smaller set, denoted $P_j$. Let $K$ be the graph containing all vertices in the parts of $P_j$. Since $\lambda^{1 \ast} \vdash m$ and $\lambda^{2} \vdash k$, we know $K$ contains $k$ vertices, and $P_j$ is an independent $\lambda^{2}$-partition of $K$. This subgraph $K$ is isomorphic to a graph $H$ with $k$ vertices, and thus $H$ maps to $G$ via the fixed map $\phi_{H}^{K}$. Since $\phi_{H}^{K}$ is an isomorphism, ${\phi_{H}^{K}}^{-1}(P_j)$ is an independent $\lambda^{2}$-partition of $H$. To satisfy $\lambda^{1} \sim \lambda^{2}$, all parts of $P^{\prime}_k$ and $P_j$ with more than one vertex must match, and for the $1$-parts, we have $\binom{n-m}{k-m}$ choices. Note that for parts with more than $1$ vertex, since we fixed the isomorphism between $H$ and $K$, the matching between non $1$-parts of $H$ and $G$ comes canonically with respect to this isomorphism. Thus, each $\lambda^{1}$-partition is counted exactly $\binom{n-m}{k-m}$ times.
\end{proof}

We proved that the relation between $c_\lambda^G$s and $c_{\lambda^\prime}^H$s is linear. Now, to calculate $c_\lambda^G$s, we shall create a matrix containing the information we need. Lemma \ref{lem:c_lamda_of_k_lamda} and Corollaries \ref{cor:lambda-mat-rank}, \ref{cor:forest-lambda-mat-rank} and \ref{cor:lamdba-mat-linear-relation} are the results concerning this matrix. 

\begin{definition} \label{def:lambda-mat}
    Let $\lambda^1, \lambda^2, \ldots, \lambda^{p(n)}$ represent all partitions of $n$, and let $G_1, G_2, \ldots, G_s$ be a family of graphs with $n$ vertices. We define the \emph{$\lambda$-matrix} of this family as follows, where $c_{i}^{j}$ indicates $c^{G_j}_{\lambda^{i}}$:
    \begin{center}
        $\begin{bmatrix}
            c_{1}^{1} & c_{2}^{1} & c_{3}^{1} & \cdots & c_{p(n)-1}^{1} & c_{p(n)}^{1} \\
            c_{1}^{2} & c_{2}^{2}& c_{3}^{2} & \cdots & c_{p(n)-1}^{2} & c_{p(n)}^{2} \\
           \vdots & \vdots & \vdots & \ddots & \vdots & \vdots \\
             c_{1}^{s-1} & c_{2}^{s-1} & c_{3}^{s-1} & \cdots & c_{p(n)-1}^{s-1} & c_{p(n)}^{s-1} \\
             c_{1}^{s} & c_{2}^{s} & c_{3}^{s} & \cdots & c_{p(n)-1}^{s} & c_{p(n)}^{s}
        \end{bmatrix}$
    \end{center}
\end{definition}

\begin{lemma} \label{lem:c_lamda_of_k_lamda}
    For the graph $K_\lambda$, where $\lambda = (\lambda_1, \lambda_2, \ldots, \lambda_\ell) \vdash n$, and for any partition $\lambda^{\prime} = (\lambda^{\prime}_1, \lambda^{\prime}_2, \ldots , \lambda^{\prime}_r) \vdash n$, it holds that $c_{\lambda^{\prime}}^{K_\lambda} \neq 0$ if and only if $\lambda^{\prime} \leq \lambda$.
\end{lemma}

\begin{proof}
    In Theorem \ref{prop:lambda-mat-vec-full-homo}, we established that $c_{\lambda}^{G}$ represents the number of independent $\lambda$-partitions of $G$. Therefore, $c_{\lambda^{\prime}}^{K_\lambda}$ is the number of $\lambda^{\prime}$-partitions in $K_\lambda$. Let $A = \left\{ P^{\prime}_1, P^{\prime}_2, \ldots, P^{\prime}_r \right\}$ be an independent $\lambda^{\prime}$-partition of $K_\lambda$ and assume $B = \left\{ P_1, P_2, \ldots, P_\ell \right\}$ is the unique independent $\lambda$-partition of $K_\lambda$.

    For any $P^{\prime}_{i}$ that contains some $v \in P_j$, if there exists any $u \in P^{\prime}_{i}$ such that $u \notin P_j$, then there must be no edges connecting $v$ to $u$, which would contradict the definition of $K_\lambda$. Consequently, if $A$ is an independent $\lambda^{\prime}$-partition of $K_\lambda$, then each part of $A$ is a subset of a part in the $\lambda$-partition $B$. This means that by taking the union of the parts of $A$ contained within each part of $B$, we can reconstruct $B$ without putting any two adjacent vertices together. Therefore, if $c_{\lambda^{\prime}}^{K_\lambda} \neq 0$, then $\lambda^{\prime} \leq \lambda$. Conversely, if $\lambda^{\prime} \leq \lambda$, we can partition the parts of the independent $\lambda$-partition to achieve $\lambda^{\prime}$, implying that $c_{\lambda^{\prime}}^{K_\lambda} \neq 0$.
\end{proof}

\begin{corollary} \label{cor:lambda-mat-rank}
    The $\lambda$-matrix of all graphs with $n$ vertices is of full rank.
\end{corollary}

\begin{proof}
    The rank of the $\lambda$-matrix is at most $p(n)$, the number of partitions of $n$. We construct $p(n)$ graphs with $n$ vertices and demonstrate that their $\lambda$-matrix is of full rank. Let $\lambda^1, \lambda^2, \ldots, \lambda^{p(n)}$ be all partitions of $n$, ordered by length from shortest to longest. The order among partitions of the same size is arbitrary. We claim that the following square matrix, in which $c^{j}_{i}$ represents $c^{K_{\lambda^j}}_{\lambda^{i}}$, is of full rank:
	\begin{center}
		$\begin{bmatrix}
			c_{1}^{1} & c_{2}^{1} & c_{3}^{1} & \cdots & c_{p(n)-1}^{1} & c_{p(n)}^{1} \\
			c_{1}^{2} & c_{2}^{2}& c_{3}^{2} & \cdots & c_{p(n)-1}^{2} & c_{p(n)}^{2} \\
			\vdots & \vdots & \vdots & \ddots & \vdots & \vdots \\
			c_{1}^{p(n)-1} & c_{2}^{p(n)-1} & c_{3}^{p(n)-1} & \cdots & c_{p(n)-1}^{p(n)-1} & c_{p(n)}^{p(n)-1} \\
			c_{1}^{p(n)} & c_{2}^{p(n)} & c_{3}^{p(n)} & \cdots & c_{p(n)-1}^{p(n)} & c_{p(n)}^{p(n)}
		\end{bmatrix}$
	\end{center}
    Observe that the main diagonal contains $c^{K_{\lambda^i}}_{\lambda^{i}}$, and for any $K_{\lambda}$, we have $c_{\lambda}^{K_{\lambda}} = 1$. Therefore, each entry on the main diagonal is equal to one. Additionally, for each diagonal element $c_{i}^{i}$, the length of the partition assigned to each row increases or remains the same as we proceed down the columns. By Lemma \ref{lem:c_lamda_of_k_lamda}, for each column $i$, any entry $c_{j}^{i}$ where $j < i$ is zero. This property makes the matrix upper triangular with all diagonal elements equal to one. Hence, the matrix is of full rank.
\end{proof}

Now, we can prove Lemma \ref{lem:dim-vec-set-graph}.

\begin{proof}[Proof of Lemma \ref{lem:dim-vec-set-graph}]
    The result follows from Corollary \ref{cor:lambda-mat-rank}, where it was established that the $\lambda$-matrix corresponding to all graphs $K_\lambda$ is of full rank. Thus, we complete the proof.
\end{proof}

\begin{corollary} \label{cor:forest-lambda-mat-rank}
    The $\lambda$-matrix of all forests with $n$ vertices is of full rank.
\end{corollary}

\begin{proof}
    The result follows directly from Theorem \ref{thm:vec-set-forest-base}.
\end{proof}

Additionally, it has been proved by Gonzalez, Orellana, and Tomba in~\cite{gonzalez2024chromatic} that the rank of $\lambda$-matrix of all Trees equals $p(n) - n + 1$. Also, a similar version of our $\lambda$-matrix is introduced in the mentioned paper and is called $n$-\emph{CSF matrix}.

\begin{corollary} \label{cor:lamdba-mat-linear-relation}
    For a graph $G$ with $n$ vertices and an integer $k < n$, if $\lambda_1, \lambda_2, \ldots, \lambda_{p(k)}$ are all the partitions of $n$ that can be $k$-reduced and $\mathcal{H}=\left\{H_1, H_2, \ldots, H_s\right\}$ is the family of all graphs with $k$ vertices. Then, if $c_{i}^j = c_{\lambda^{i \prime}}^{H_j}$ and $M_i = \binom{n-m_i}{k-m_i}$, where $ \lambda^{i\ast} \vdash m_i$ and $\lambda ^{i \prime}$ is the $k$-reduced form of $\lambda ^ i$, we have:

 \begin{center}
        $\begin{bmatrix}
            \binom{G}{H_1} & \binom{G}{H_2} & \cdots & \binom{G}{H_s} 
        \end{bmatrix}
        \begin{bmatrix}
            c_{1}^{1} & c_{2}^{1} & c_{3}^{1} & \cdots & c_{p(n)-1}^{1} & c_{p(n)}^{1} \\
            c_{1}^{2} & c_{2}^{2}& c_{3}^{2} & \cdots & c_{p(n)-1}^{2} & c_{p(n)}^{2} \\
           \vdots & \vdots & \vdots & \ddots & \vdots & \vdots \\
             c_{1}^{s-1} & c_{2}^{s-1} & c_{3}^{s-1} & \cdots & c_{p(n)-1}^{s-1} & c_{p(n)}^{s-1} \\
             c_{1}^{s} & c_{2}^{s} & c_{3}^{s} & \cdots & c_{p(n)-1}^{s} & c_{p(n)}^{s}
        \end{bmatrix}
        = 
        \begin{bmatrix}
            M_1 c_1 & M_2 c_2 & \cdots & M_{p(k)} c_{p(k)}
        \end{bmatrix}.$
    \end{center}
\end{corollary}

\begin{proof}
    The result follows directly from Theorem \ref{prop:lambda-mat-vec-full-homo}.
\end{proof}

Returning to the reconstruction conjecture, an intriguing question arises: is it possible to reconstruct a tree’s deck of cards using $\lambda$-matrices? Given that $\lambda$-matrices are of full rank, one might initially hope that they could allow for complete recovery of the cards. However, this is not the case. The number of forests and trees with $n$ vertices far exceeds $p(n)$, the dimensions of the corresponding spaces. Consequently, it is impossible to recover data from the larger space using information derived from the smaller space. Nevertheless, the full rank of the matrix imposes constraints on the possible configurations of the cards, making this an interesting avenue for further investigation.


\section*{Acknowledgments} \label{section:akcn}

The authors would like to express their gratitude to Farid Aliniaeifard for his invaluable supervision and assistance throughout this project, which would have been impossible without his help and insights.  We also extend our appreciation to José Aliste-Prieto and Shu Xiao Li for their constructive comments, which significantly contributed to improving the clarity of this paper.

\pagebreak
\printbibliography

\end{document}